\documentclass[11pt]{article}
\usepackage[a4paper, total={6in, 8in}]{geometry}
\usepackage{enumerate}
\usepackage{amsmath}
\usepackage{amssymb,latexsym}
\usepackage{amsthm}
\usepackage{color}
\usepackage{cancel}
\usepackage{graphicx}
\usepackage{cite}
\usepackage{longtable}
\makeatletter
\usepackage{amscd}

\newtheorem{theorem}{Theorem}[section]

\newtheorem{corollary}[theorem]{Corollary}

\newtheorem{proposition}[theorem]{Proposition}

\newtheorem{remark}[theorem]{Remark}

\newtheorem{question}[theorem]{Question}

\newcommand{\F}{{\mathbb F}}

\newcommand{\N}{\mathrm{N}}

\title{On the infiniteness of a family of APN functions}
\author{Daniele Bartoli\thanks{Dipartimento di Matematica e Informatica, Universit\`a degli Studi di Perugia,  Perugia, Italy. daniele.bartoli@unipg.it}, 
Marco Calderini\thanks{Department of Informatics, University of Bergen, Bergen, Norway. marco.calderini@uib.no}, 
Olga Polverino\thanks{Dipartimento di Matematica e Fisica, Universit\`a degli Studi della Campania ``Luigi Vanvitelli'', Caserta, Italy.
olga.polverino@unicampania.it},
and
Ferdinando Zullo\thanks{Dipartimento di Matematica e Fisica, Universit\`a degli Studi della Campania ``Luigi Vanvitelli'', Caserta, Italy.
ferdinando.zullo@unicampania.it}
}
\date{ }

\begin{document}

\maketitle

\begin{abstract}
APN functions play a fundamental role in cryptography against attacks on block ciphers. Several families of quadratic APN functions have been proposed in the recent years, whose construction relies on the existence of specific  families of polynomials. A key question connected with such constructions is to determine whether such APN functions exist for infinitely many dimensions or not. 

In this paper we consider a family of functions recently introduced by Li et al. in 2021 showing that for any dimension $m\geq 3$ there exists an APN function belonging to such a family. 

Our main result is proved by a combination of different techniques arising from both algebraic varieties over finite fields connected with linearized permutation rational functions and {partial vector space partitions}, together with investigations on the kernels of linearized polynomials. 
\end{abstract}

\section{Introduction}

Let $\mathbb{F}_{2^n}$ be the finite fields with $2^n$ elements, and denote by $\mathbb{F}_{2^n}^*$ its multiplicative group. Given a function $f:\mathbb{F}_{2^n}\to\mathbb{F}_{2^n}$, it is interesting to understand how many solutions $x$  the equation 
\begin{equation}\label{eq:fondamentale}
f(x+a)+f(x)=b
\end{equation}
has, for any $a\in\mathbb{F}_{2^n}^*$ and $b\in\mathbb{F}_{2^n}$. Note that if $x$ is a solution, then so it is $x+a$.
A function $f$ is said to be \emph{almost perfect nonlinear} (APN) if there are always exactly zero or two solutions to \eqref{eq:fondamentale}. The function $f(x+a)+f(x)$ is called the  derivative of $f$ in the direction $a$. Thus, an APN function is a function whose derivatives yield two-to-one maps over $\mathbb{F}_{2^n}$.

APN functions were introduced by Nyberg in \cite{Nyb93}, in the context of cryptography,  as the mappings with highest resistance to differential cryptanalysis \cite{diff}, one of the most efficient attacks that can be employed against block ciphers.

APN functions are also interesting from a theoretical point of view, as they correspond to optimal objects within different areas of mathematics and computer science.

For instance, they have been  constructed in connection with several  combinatorial  and geometrical objects, such as semi-biplanes \cite{CH99} and dual-hyperovals \cite{DE14}. In this context these mappings are also called semi-planar \cite{DO68}. Another application of APN functions is related with the construction of error correcting codes, since each APN function yields a double error correcting BCH-like code.

Equivalence issues play an important role in the study of such functions. The above connection with BCH codes also provides an equivalence definition between APN functions: two APN functions are said to be inequivalent if the (extended) BCH-like codes obtained from them are inequivalent codes (see \cite{dillon} for more details). This  relation is called CCZ-equivalence \cite{CCZ}, and it is the most general equivalence relation preserving the APN property.

In the last years, several families of (quadratic) APN  functions (see \cite{BCV20} or \cite{Car20} for a recent list of inequivalent APN families) were constructed. For some of these families the APN property is connected with the existence of  polynomials having specific features; see i.e.  \cite{BC08,Tan19,BCCCV20,BCCCV21}.  It is therefore crucial to understand whether APN functions coming from these constructions  exist for  infinitely many  dimensions or not.

For example, the hexanomials  
$$
d x^{2^s({2^m+1})}+x^{({2^m+1})}+x^{2^s+1}+x^{2^m({2^s+1})}+ c x^{2^{m+s}+1}+ c^{2^m}x^{{2^s+2^m}} \in \mathbb{F}_{2^{2m}}[x],
$$
where $d\notin \mathbb{F}_{2^m}$ and $\gcd(s,m)=1$, is APN if and only if $x^{2^s+1}+ c x^{2^s} + c ^{2^m}x+1=0$
has no solution $x$ such that $x^{2^m +1} = 1$; see \cite{BC08}. The existence of polynomials satisfying this last condition was verified by a computer in \cite{BC08} whenever $6\le 2m\le 500$.
The existence, for  infinite many values of $n$, of instances of APN functions belonging to this family has been investigated in several works \cite{Brack14,blu,gologlu}.

In this paper we consider the  new family of quadratic APN functions,  recently introduced in \cite{newAPN}, that  generalizes the one given in \cite{BBMM}. Let $s$ and $m$ be integers such that $\gcd(s,m)=1$. The mapping defined over $\mathbb{F}_{2^{3m}}$
\begin{equation}\label{eq:APNfun}
(x^{2^{m+s}} + \mu x^{2^s} + x)^{2^m+1}+vx^{2^m+1},
\end{equation}
where $\mu \in \mathbb{F}_{2^{3m}}$ satisfies $\N_{2^{3m}/2^m}(\mu):=\mu^{2^{2m}+2^m+1}\neq 1$ and $v\in \mathbb{F}_{2^m}^*$, is APN whenever $f_\mu^{(s)}(x):=x^{2^{m+s}} + \mu x^{2^s} + x$ permutes $\mathbb{F}_{2^{3m}}$. In \cite{newAPN}, the authors checked the existence of such an element $\mu$ for $3\le m\le 8$ and raised the following question.

\begin{question}\cite{newAPN}\label{question}
Let $m\ge3$ and $s\ge 1$, such that  $\gcd(s,m)=1$. Does there exist $\mu\in \mathbb{F}_{2^{3m}}$ such that $\N_{2^{3m}/2^m}(\mu)\neq 1$ and $f_\mu^{(s)}(x)$ is a permutation polynomial?
\end{question}

The main achievement of this paper is the proof of the existence for all $m\geq 3$ of suitable $s$ and $\mu$, $\N_{2^{3m}/2^m}(\mu)\neq 1$, for which the polynomial \eqref{eq:APNfun} is APN. This is done combining techniques from both algebraic geometry over finite fields and {partial vector space partitions}. A key tool in our machinery  is the investigation of the kernel of $2$-linearized polynomials of the type $f_\mu^{(s)}(x):=x^{2^{m+s}} + \mu x^{2^s} + x \in \mathbb{F}_{2^{3m}}[x]$ (see Section \ref{sec:1}), and determination of sufficient conditions involving $m$ and $s$  for the existence of elements $\mu$ with $\N_{2^{3m}/2^m}(\mu)\ne 1$  such that the polynomial $f_\mu^{(s)}(x)$ admits only one root; see Section \ref{sec:2}. This provides a (partial) positive answer to Question \ref{question} (see Corollary \ref{cor:ms}) and, thus, to the existence of APN functions belonging to the family  \eqref{eq:APNfun}.

A fruitful  connection to algebraic varieties over finite fields is provided in  Section \ref{sec:3}. Estimating the number of  $\mathbb{F}_{2^m}$-rational points of suitable three-dimensional varieties, we are able to prove, for the case $s=1$, that for any $m\ge 3$ there always exists an element $\mu\in \mathbb{F}_{2^{3m}}$ such that $f_\mu^{(1)}(x)$ is a permutation and $\N_{2^{3m}/2^m}(\mu)\neq 1$. This provides a positive answer to Question \ref{question} for the case  $s=1$.

\section{Bounds on the dimension of the kernel of $f_{\mu}^{(s)}(x)$}\label{sec:1}
{In this section we collect a few results on the kernel of linearized polynomials $f_{\mu}^{(s)}(x)$ that will be used in the sequel.}
Let $q$ be a prime power and let $m$ be a positive integer. A \emph{linearized polynomial}, or $q$-\emph{polynomial}, over $\F_{q^m}$ is a polynomial of the form
\[f(x)=\sum_{i=0}^t a_i x^{q^i},\]
where $a_i\in \F_{q^m}$, $t$ is a positive integer.
We denote by $\mathcal{L}_{m,q}$ the set of all $q$-polynomials over $\F_{q^m}$.
The $\F_q$-linear maps of $\F_{q^m}$ can be identified with the polynomials in ${\mathcal{L}}_{m,q}$ of degree at most $q^{m-1}$. The \emph{kernel}, of a polynomial $f(x)\in\mathcal{L}_{m,q}$ will be denoted by $\ker(f):=\{x\in\mathbb{F}_{q^m}\,:\,f(x)=0\}$.

{ For an element $\mu \in \F_{2^{3m}}^*$, denote by $P_{\mu}$ the one-dimensional vector $\F_{2^{3m}}$-subspace $\langle (1,\mu) \rangle_{\F_{2^{3m}}}$ of $(\mathbb{F}_{2^{3m}})^2$.}

\begin{proposition}\label{prop:weightpointperm}
Consider
\[U_{s}=\{ (x^{2^s},x^{2^{m+s}}+x) \colon x \in \F_{2^{3m}} \}.\]
Then $\dim_{\F_2}(\ker(f_{\mu}^{(s)}))=\dim_{\F_2}(U_{s}\cap P_{\mu})$.
In particular, $f_{\mu}^{(s)}(x)$ is a permutation if and only if $\dim_{\F_2}(U_{s}\cap P_{\mu})=0$.
\end{proposition}

We recall the following result from \cite{polverino2020number}.

\begin{theorem}\cite[Theorem 1.5, Corollary 5.3]{polverino2020number}\label{th:PZ}
Let
\[ f(x)=-x+b_0x^{\sigma}+b_1x^{\sigma q^m}+b_2x^{\sigma q^{2m}}+\ldots+b_{t-1}x^{\sigma q^{m(t-1)}} \in \mathcal{L}_{mt,q}, \]
where $\sigma \in \mathrm{Aut}(\F_{q^{mt}})$ such that $\sigma|_{\F_{q^m}}\colon\F_{q^m}\rightarrow\F_{q^m}$ has order $m$.
Let $G(x)$ be the $q^m$-polynomial such that $f(x)=(G\circ\sigma)(x)-x$, i.e. $G(x)=\sum_{i=0}^{t-1} b_ix^{q^{mi}}$ and $H(x)=(G\circ\sigma)(x)$.
Then
\begin{equation}\label{eq:boundG}   
\dim_{\F_q} (\ker(f)) \leq t- \dim_{\F_{q^m}} \ker (G).
\end{equation}
Moreover,
$\dim_{\F_q}(\ker(f))=h$ if and only if
\begin{equation*}
\dim_{\F_{q^m}}(\ker(H^m-\mathrm{id}))=h.
\end{equation*}
\end{theorem}

As a consequence we obtain the following.

\begin{corollary}\label{cor:8}
Let $f_{\mu}^{(s)}(x)=x^{2^{m+s}}+\mu x^{2^s}+x \in \F_{2^{3m}}[x]$ with $\gcd(s,m)=1$.
For every $\mu \in \F_{2^{3m}}$,
\[ \dim_{\F_2}(U_s \cap P_{\mu})=\dim_{\F_2}(\ker(f_{\mu}^{(s)}))\leq 3. \]
Also
\[ \dim_{\F_2}(\ker(f_{0}^{(s)}))=\gcd(3,s)\in \{1,3\}. \]
Moreover, if $\N_{q^{3m}/q^m}(\mu)=1$ then 
\begin{equation}\label{eq:norm1} \dim_{\F_2}(\ker(f_{\mu}^{(s)}))\leq 2. \end{equation}
\end{corollary}
\begin{proof}
Since $\gcd(s,m)=1$, $\sigma \colon x \in \F_{2^{3m}} \mapsto x^{2^s} \in \F_{2^{3m}}$ is an automorphism of $\F_{2^{3m}}$ such that $\sigma|_{\F_{2^m}}\colon\F_{2^m}\rightarrow\F_{2^m}$ has order $m$. Theorem \ref{th:PZ} applied to $f_{\mu}^{(s)}(x)$ yields that for every $\mu \in \F_{2^{3m}}$
\[ \dim_{\F_2}(\ker(f_{\mu}^{(s)}))\leq 3. \]
Therefore, by Proposition \ref{prop:weightpointperm} it follows that $\dim_{\F_2}(U_s\cap P_{\mu})=\dim_{\F_2}(\ker(f_{\mu}^{(s)}))\leq 3$.

When $\mu=0$, $f_{0}^{(s)}(x)=x^{2^{m+s}}+x$ and $\dim_{\F_2}(\ker(f_{0}^{(s)}))=\gcd(3m,s+m)$ and the assertion then follows.

If $\N_{2^{3m}/2^m}(\mu)=1$ then $G(x)=\mu x+x^{2^{m}}$ has $2^m$ roots in $\F_{2^{3m}}$ and by  Theorem \ref{th:PZ}\eqref{eq:boundG} we obtain
\[ \dim_{\F_2}(\ker(f_{\mu}^{(s)}))\leq 2. \]
\end{proof}

\section{A partial answer via the dimension of the kernel of $f_{\mu}^{(s)}(x)$}\label{sec:2}

{This section provides a positive answer to specific instances of Question \ref{question}, investigating the kernel of the polynomials $f_{\mu}^{(s)}(x)$.}

\begin{proposition}\label{prop:dim1}
Let $s$ be a positive integer such that $s+m$ is coprime with $3m$.
For any $\eta \in \F_{2^{3m}}$ with $\eta \ne 1$ there exist a unique $x_0$ and a unique $y_0$ in $\F_{2^{3m}}^*$ such that 
\[ (x_0^{2^s},x_0+x_0^{2^{s+m}})=(y_0,\eta y_0^{2^m}). \]
\end{proposition}
\begin{proof}
Let $\eta \in \F_{2^{3m}}$ with $\eta \ne 1$.
Our aim is to find a unique $x_0$ and a unique $y_0$ in $\F_{2^{3m}}^*$ such that
\[ (x_0^{2^s},x_0^{2^{s+m}}+x_0)=(y_0,\eta y_0^{2^m}). \]
From the above identity one gets
\[ y_0=x_0^{2^s} \,\,\,\text{and}\,\,\, \eta y_0^{2^m}=x_0^{2^{s+m}}+x_0, \]
from which we obtain
\[ x_0^{2^{s+m}-1}=\frac{1}{\eta-1}. \]
Since $\gcd(s+m,3m)=1$ and $\N_{2^{3m}/2}\left(\frac{1}{\eta-1}\right)=1$, the above equation admits exactly one solution.
\end{proof}

{In what follows, for an $\eta\in \F_{2^{3m}}^*$ let $W_{\eta}$ denote the vector space $\{ (y,\eta y^{2^m}) \colon y \in \F_{2^{3m}} \}$, seen as three-dimensional vector space over $\mathbb{F}_{2^m}$.}

\begin{proposition}\label{prop:relationni}
Let $s$ be a positive integer such that $s+m$ is coprime with $3m$, $\alpha \in \F_{2^m}\setminus \{1\}$, and  
\[ n_i(\alpha)=\{P_{\mu} \colon \mu \in \F_{2^{3m}},\N_{2^{3m}/2^m}(\mu)=\alpha,\,\, \dim_{\F_2}(P_\mu\cap U_s)=i\}\] 
for any $i \in \{0,1,2,3\}$.
Then 
\begin{equation*}
n_0(\alpha)=2n_2(\alpha)+6n_3(\alpha). \end{equation*}
\end{proposition}
\begin{proof}
Clearly,
\begin{equation*}
    \bigcup_{\substack{\mu \in \F_{2^{3m}} \\ \N_{2^{3m}/2^m}(\mu)=\alpha}} 
    P_{\mu}=\bigcup_{\substack{\eta \in \F_{2^{3m}}\\ \N_{2^{3m}/2^m}(\eta)=\alpha}} W_{\eta},
\end{equation*}
and 
\[ |\{P_{\mu} \colon \mu \in \F_{2^{3m}}, \N_{2^{3m}/2^m}(\mu)=\alpha\}|=|\{W_{\eta} \colon \eta \in \F_{2^{3m}}, \N_{2^{3m}/2^m}(\eta)=\alpha\}|=2^{2m}+2^m+1. \]

By Corollary \ref{cor:8}, we have
\[ \dim_{\F_2}(P_\mu\cap U_s)\leq 3 \]
for any $\mu \in \F_{2^{3m}}$, so that 
\begin{equation}\label{eq:2} n_0(\alpha)+n_1(\alpha)+n_2(\alpha)+n_3(\alpha)=2^{2m}+2^m+1. 
\end{equation}

By Proposition \ref{prop:dim1} we have
\[ \dim_{\F_2}(W_\eta \cap U_s)=1, \]
for every $\eta \in \F_{2^{3m}}\setminus\{1\}$.
Since
\[ U_s \cap \left( \bigcup_{\substack{\eta \in \F_{2^{3m}}\\ \N_{2^{3m}/2^m}(\eta)=\alpha}} P_{\mu} \right)=U_s \cap \left( \bigcup_{\substack{\eta \in \F_{2^{3m}}\\ \N_{2^{3m}/2^m}(\eta)=\alpha}} W_{\eta} \right) \]
we get
\[ \left| (U_s \setminus\{(0,0)\}) \cap \left( \bigcup_{\substack{\eta \in \F_{2^{3m}}\\ \N_{2^{3m}/2^m}(\eta)=\alpha}} P_{\mu} \right) \right|= \sum_{\substack{\eta \in \F_{2^{3m}}\\ \N_{2^{3m}/2^m}(\eta)=\alpha}} |W_\eta \cap (U_s\setminus\{(0,0)\})|=2^{2m}+2^m+1. \]
Hence,
\begin{equation}\label{eq:1} \left| (U_s \setminus\{(0,0)\}) \cap \left( \bigcup_{\substack{\mu \in \F_{2^{3m}}\\ \N_{2^{3m}/2^m}(\mu)=\alpha}} P_{\mu} \right) \right|=n_1(\alpha)+3n_2(\alpha)+7n_3(\alpha)=2^{2m}+2^m+1. 
\end{equation}

By \eqref{eq:1} and \eqref{eq:2} we get the assertion.
\end{proof}

\begin{remark}\label{rem:point}
By {Proposition \ref{prop:relationni}} the existence of an element $\mu \in \F_{2^{3m}}$ such that $\alpha=\N_{2^{3m}/2^m}(\mu)\notin\{0,1\}$ and $\dim_{\F_2}(\ker(f_{\mu}^{(s)}))\geq 2$ implies the existence of an element $\eta \in \F_{2^{3m}}$ such that $\N_{2^{3m}/2^m}(\eta)=\alpha=\N_{2^{3m}/2^m}(\mu)\notin\{0,1\}$ and $f_{\eta}^{(s)}(x)$ permutes $\F_{2^{3m}}$.
\end{remark}

\begin{proposition}\label{prop:recursive}
Let $s$ be a positive integer such that $s+m$ is coprime with $3m$.
For every $\mu \in \F_{2^{3m}}$,
\[ \dim_{\F_2}(\ker(f_\mu^{(s)}))=\dim_{\F_{2^m}}(\ker(H^m-\mathrm{id})), \]
where \[ H^i(x)=h_{0i}x^{2^{is}}+h_{1i}x^{2^{is+m}}+h_{2i}x^{2^{is+2m}}, \]
with
\[ \left\{
\begin{array}{lll}
h_{01}=\mu,\\
h_{11}=1,\\
h_{21}=0,
\end{array}
\right. \]
and
\[ \left\{
\begin{array}{lll}
h_{0i+1}=\mu h_{0i}^{2^s}+h_{2i}^{2^{s+m}},\\
h_{1i+1}=\mu h_{1i}^{2^s}+h_{0i}^{2^{s+m}},\\
h_{2i+1}=\mu h_{2i}^{2^s}+h_{1i}^{2^{s+m}},
\end{array}
\right. \]
for any $i\geq 1$.
\end{proposition}
\begin{proof}
The proof immediately follows from Theorem \ref{th:PZ} with $H(x)=\mu x^{2^s}+x^{2^{s+m}}$.
\end{proof}

\begin{theorem}\label{th:cases}
Let $\overline{m},\overline{s},m,t,j,s \in \mathbb{N}$ and let $\mu \in \F_{2^{\overline{m}}}$ be a root of $g(x)$ as in Table \ref{table}.
Let $\overline{H}(x)=\mu x^{2^{\overline{s}}}+x^{2^{\overline{s}+\overline{m}}} \in \mathcal{L}_{3\overline{m},2}$ and let $H(x)=\mu x^{2^{s}}+x^{2^{s+ m}} \in \mathcal{L}_{3 m,2}$. 
Then
\[ \overline{H}^{\overline{m}}(x)=x^{2^{\overline{m}(\overline{s}+j)}}\,\,\,\text{and}\,\,\, H^{m}(x)=x^{2^{m(s+t)}}.\]
Hence, if $3\mid s+t$, $H^m(x)=x$ and $\dim_{\F_2}(\ker(f_{\mu}^{(s)}))=3$.
In particular, $\N_{2^{3m}/2^m}(\mu)\ne 1$.
\begin{table}[ht!]
    \centering
    \begin{tabular}{|c|c|c|c|}
    \hline
    $(\overline{m},\overline{s})$ & $g(x)$ & $j$& Conditions \\
    \hline
    \hline
     \begin{tabular}{c} $(3,1)$\\ $(3,2)$\end{tabular} & $x^3 + x + 1$ & $1$ & \begin{tabular}{c} $m=3t$, $3\nmid t$, $\gcd(m,s)=1$\\ $3\mid (s-\overline{s})$, $3\mid (s+t)$\end{tabular}  \\
    \hline
    \begin{tabular}{c} $(3,1)$\\ $(3,2)$\end{tabular} & $x^3 + x^2 + 1$ & $2$ & \begin{tabular}{c} $m=3t$,  $3\nmid t$, $\gcd(m,s)=1$\\ $3\mid (s-\overline{s})$, $3\mid (s+2t)$\end{tabular}  \\
    \hline
     \begin{tabular}{c}$(4,1)$\\$(4,3)$ \end{tabular} & $x^4 + x^3 + x^2 + x + 1$ & $2$ & \begin{tabular}{c} $m=4t$, $3\nmid t$, $\gcd(m,s)=1$\\ $4\mid (s-\overline{s})$, $3\mid (s+2t)$\end{tabular}  \\
     \hline
    \begin{tabular}{c}$(5,1)$\\$(5,2)$\\$(5,3)$\\$(5,4)$\\ \end{tabular} & $x^5 + x^4 + x^3 + x^2 + 1$ & $1$ & \begin{tabular}{c} $m=5t$, $3\nmid t$, $\gcd(m,s)=1$\\ $5\mid (s-\overline{s})$, $3\mid (s+t)$\end{tabular}  \\
     \hline  
    \begin{tabular}{c} $(7,1)$\\ $(7,2)$\\$(7,3)$\\$(7,4)$\\$(7,5)$\\$(7,6)$\\\end{tabular} & $x^7 + x^5 + x^3+x+ 1$ & $2$ & \begin{tabular}{c} $m=7t$, $3\nmid t$, $\gcd(m,s)=1$\\ $7\mid (s-\overline{s})$, $3\mid (s+2t)$\end{tabular}  \\
        \hline  
    \begin{tabular}{c} $(9,1)$\\ $(9,2)$\\$(9,4)$\\$(9,5)$\\$(9,7)$\\$(9,8)$\\\end{tabular} & $x^9 + x^4 + 1$ & $1$ & \begin{tabular}{c} $m=9t$, $3\nmid t$, $\gcd(m,s)=1$\\ $9\mid (s-\overline{s})$, $3\mid (s+t)$\end{tabular}  \\
    \hline
      \begin{tabular}{c} $(9,1)$\\ $(9,2)$\\$(9,4)$\\$(9,5)$\\$(9,7)$\\$(9,8)$\\\end{tabular} & $x^9 + x^8 + 1$ & $2$ & \begin{tabular}{c} $m=9t$, $3\nmid t$, $\gcd(m,s)=1$\\ $9\mid (s-\overline{s})$, $3\mid (s+2t)$\end{tabular}  \\
    \hline
    \end{tabular}
    \caption{Assumptions of Theorem \ref{th:cases}}
    \label{table}
\end{table}
\end{theorem}
\begin{proof}
Suppose that $\overline{m}=3$.
Since the coefficients of $\overline{H}^3(x)$ are
\[ \left\{
\begin{array}{lll}
\overline{h}_{03}=\mu^{1+2^{\overline{s}}+2^{2\overline{s}}}+1,\\
\overline{h}_{13}=\mu^{1+2^{\overline{s}}}+\mu^{1+2^{2\overline{s}+\overline{m}}}+\mu^{2^{\overline{s}+\overline{m}}+2^{2\overline{s}+\overline{m}}},\\
\overline{h}_{23}=\mu+\mu^{2^{\overline{s}+\overline{m}}}+\mu^{2^{2\overline{s}+2\overline{m}}},
\end{array}
\right. \]
if $\mu \in \F_{2^3}\setminus\{0,1\}$ is a root of $x^3+x+1$ (respectively of $x^3+x^2+1$) then $\overline{h}_{03}=\overline{h}_{23}=0$ and $\overline{h}_{13}=1$ (respectively $\overline{h}_{03}=\overline{h}_{13}=0$ and $\overline{h}_{23}=1$), i.e.
\[ \overline{H}^3(x)=x^{2^{3\overline{s}+j\overline{m}}}=x^{2^{\overline{m}(\overline{s}+j)}}, \]
where $j=1$ if $\mu^3+\mu+1=0$ and $j=2$ if $\mu^3+\mu^2+1=0$.
Note that from the assumptions on $m$ and $s$ of Table \ref{table}, it follows that  $\overline{h}_{03}=h_{03}$, $\overline{h}_{13}=h_{13}$ and $\overline{h}_{23}=h_{23}$, that is
\[ H^3(x)=x^{2^{3s+jm}} \]
so that 
\[ H^m(x)=H^{3t}(x)=x^{2^{m(s+jt)}}. \]
Since $3 \mid s+jt$, we obtain that $H^m(x)-x$ is the zero polynomial and hence by Proposition \ref{prop:recursive} we have $\dim_{\F_2}(\ker(f_{\mu}^{(s)}(x)))=3$.
In the other cases, MAGMA computations show that if $\mu$ is a root of $g(x)$ as in Table \ref{table}, we get $\overline{h}_{0\overline{m}}=0$ and $\{\overline{h}_{1\overline{m}},\overline{h}_{2\overline{m}}\}=\{0,1\}$ and $\overline{h}_{j,\overline{m}}=1$, so that
\[ \overline{H}^{\overline{m}}(x)=x^{2^{\overline{m}s+j\overline{m}}}. \]
Now, as in the previous case, under the assumptions on $m,s$ and $t$ of Table \ref{table} we get
\[ H^m(x)=x, \]
i.e. $\dim_{\F_2}(\ker(f_{\mu}^{(s)}))=3$.
Finally, by \eqref{eq:norm1} of Corollary \ref{cor:8} we also get that $\N_{2^{3m}/2^m}(\mu)\ne 1$.
\end{proof}

Taking into account Remark \ref{rem:point} and of Theorem \ref{th:cases} we obtain the following partial answer to Question \ref{question}.

\begin{corollary}\label{cor:ms}
Let $s$ be a positive integer such that $\gcd(s+m,3m)=1$.
There exists $\mu\in \mathbb{F}_{2^{3m}}$, $\N_{2^{3m}/2^{m}}(\mu)\neq 1$, such that  $f_{\mu}^{(s)}(x)$ is a permutation  in the following cases:
\begin{itemize}
    \item $m=3t$ and $t\not\equiv 0 \pmod{3}$ for any $s$ such that $\gcd(s,m)=1$;
    \item $m=4t$, $t\not\equiv 0 \pmod{3}$ and $s+2t \equiv 0 \pmod{3}$;
     \item $m=5t$, $t \not\equiv 0 \pmod{3}$ and $s+t\equiv 0\pmod{3}$;
     \item $m=7t$, $t \not\equiv 0 \pmod{3}$ and $s+2t\equiv 0\pmod{3}$;
     \item $m=9t$ and $t \not\equiv 0 \pmod{3}$ for any $s$ such that $\gcd(s,m)=1$.
\end{itemize}
\end{corollary}

\section{A connection with algebraic varieties}\label{sec:3}
{This section provides a positive answer to Question \ref{question} for the case $s=1$, exploiting a connection with algebraic varieties over finite fields.}

{Let $q$ denote $2^m$. As a notation,  $\mathbb{A}^r(\mathbb{F})$ and $\mathbb{P}^r(\mathbb{F})$ denote the affine and the projective $r$-dimensional space over the field $\mathbb{F}$, respectively.  We start with the following observation.}
\begin{proposition}
Let $f_\mu^{(s)}(x):= x^{2^sq}+\mu x^{2^s}+x$, $\mu \in \mathbb{F}_{q^3}$. Then
$$\{\mu \in \mathbb{F}_{q^3} : f_\mu^{(s)}(x) \textrm{ is a permutation}  \}= \mathbb{F}_{q^3}\setminus \left\{ \frac{x^{2^sq}+x}{x^{2^s}} : x \in \mathbb{F}_{q^3}^*\right\}.$$
\end{proposition}
\begin{proof}
Since  $f_\mu^{(s)}(x)$ is linearized, $f_\mu^{(s)}(x)$ is a PP if and only if its kernel is $\{0\}$. Now, $\mu\in \left\{ \frac{x^{2^sq}+x}{x^{2^s}} : x \in \mathbb{F}_{q^3}^*\right\}$ if and only if $f_\mu^{(s)}(x)=0$ has an extra solution $\overline{x}\neq 0$, that is to say $\{0\}\subsetneq \ker (f_\mu^{(s)})$. The claim follows.
\end{proof}

In order to determine the value set of the rational function $\frac{x^{2^sq}+x}{x^{2^s}}$, we use an approach based on algebraic varieties over finite fields. In particular, we need a lower bound on the number of pairs $(\overline{x}, \overline{y})\in (\mathbb{F}_{q^3}^*)^2$, $\overline{x}\neq \overline{y}$, such that 
$$\frac{\overline{x}^{2^sq}+\overline{x}}{\overline{x}^{2^s}}=\frac{\overline{y}^{2^sq}+\overline{y}}{\overline{y}^{2^s}}.$$
In other words, a key point in our argument is to provide a lower bound on the number of $\mathbb{F}_{q^3}$-rational point of the curve  

\begin{equation*}
\mathcal{C}_s \ : \ \frac{Y^{2^s}(X^{2^sq}+X)+X^{2^s}(Y^{2^sq}+Y)}{X+Y}=0 \subset \mathbb{P}^{2}(\mathbb{F}_{q^3}).
\end{equation*}

Unfortunately, due to the high degree of $\mathcal{C}_s$ with respect to the size of the ground field $\mathbb{F}_{q^3}$, the investigation of absolutely irreducible $\mathbb{F}_{q^3}$-rational components in $\mathcal{C}_s$ is useless to our goal. 

We use a slightly different approach to get the desired lower bound. Consider a basis $\{\xi,\xi^q,\xi^{q^2}\}$ of $\mathbb{F}_{q^3}$ over $\mathbb{F}_q$. Denote by 
\begin{eqnarray*}
\overline{U_0} &:=& X_0\xi+X_1 \xi^q+X_2 \xi^{q^2};\\
\overline{U_1} &:=& X_0\xi^q+X_1 \xi^{q^2}+X_2 \xi;\\
\overline{U_2} &:=& X_0\xi^{q^2}+X_1 \xi+X_2 \xi^{q};\\
\overline{V_0} &:=& Y_0\xi+Y_1 \xi^q+Y_2 \xi^{q^2};\\
\overline{V_1} &:=& Y_0\xi^q+Y_1 \xi^{q^2}+Y_2 \xi;\\
\overline{V_2} &:=& Y_0\xi^{q^2}+Y_1 \xi+Y_2 \xi^{q}.
\end{eqnarray*}

Consider the variety  $\mathcal{V}_s$ defined by 
\begin{equation*}
\begin{cases}
\frac{\overline{V_0}^{2^s}(\overline{U_1}^{2^s}+\overline{U_0})+\overline{U_0}^{2^s}(\overline{V_1}^{2^s}+\overline{V_0})}{\overline{U_0}+\overline{V_0}}=0\\
\frac{\overline{V_1}^{2^s}(\overline{U_2}^{2^s}+\overline{U_1})+\overline{U_1}^{2^s}(\overline{V_2}^{2^s}+\overline{V_1})}{\overline{U_1}+\overline{V_1}}=0\\
\frac{\overline{V_2}^{2^s}(\overline{U_0}^{2^s}+\overline{U_2})+\overline{U_2}^{2^s}(\overline{V_0}^{2^s}+\overline{V_2})}{\overline{U_2}+\overline{V_2}}=0.
\end{cases}
\end{equation*}
It is readily seen that $\mathcal{V}_s$ is $\mathbb{F}_q$-rational and there is a bijection between $\mathbb{F}_q$-rational points $(x_0,x_1,x_2,y_0,y_1,y_2)\in \mathcal{V}_s$ and $\mathbb{F}_{q^3}$-rational points $(x_0\xi+x_1\xi^q+x_2\xi^{q^2},y_0\xi+y_1\xi^q+y_2\xi^{q^2})\in \mathcal{C}_s$. Also the variety $\mathcal{V}_s$ is $\mathbb{F}_{q^3}$-equivalent to $\mathcal{V}_s^{\prime}$ defined by 
\begin{equation*}
\begin{cases}
\frac{{V_0}^{2^s}({U_1}^{2^s}+{U_0})+ {U_0}^{2^s}( {V_1}^{2^s}+ {V_0})}{ {U_0}+ {V_0}}=0\\
\frac{ {V_1}^{2^s}( {U_2}^{2^s}+ {U_1})+ {U_1}^{2^s}( {V_2}^{2^s}+ {V_1})}{ {U_1}+ {V_1}}=0\\
\frac{ {V_2}^{2^s}( {U_0}^{2^s}+ {U_2})+ {U_2}^{2^s}( {V_0}^{2^s}+ {V_2})}{ {U_2}+ {V_2}}=0,
\end{cases}
\end{equation*}
via the linear isomorphism \begin{eqnarray*}\!L(X_0,X_1,X_2,Y_0,Y_1,Y_2)\!\!\!\!\!\!&=\!\!\!\!\!\!&(X_0\xi+X_1\xi^q+X_2\xi^{q^2},X_0\xi^q+X_1\xi^{q^2}+X_2\xi,X_0\xi^{q^2}+X_1\xi+X_2\xi^{q},\\ &&Y_0\xi+Y_1\xi^q+Y_2\xi^{q^2},Y_0\xi^q+Y_1\xi^{q^2}+Y_2\xi,Y_0\xi^{q^2}+Y_1\xi+Y_2\xi^{q}).
\end{eqnarray*}
Through this equivalence, $\mathbb{F}_q$-rational components of $\mathcal{V}_s$ are mapped to components of $\mathcal{V}_{s}^{\prime}$ fixed by 

\begin{eqnarray*}
\phi&: &\mathbb{F}_{q^3}[U_0,U_1,U_2,V_0,V_1,V_2] \to \mathbb{F}_{q^3}[U_0,U_1,U_2,V_0,V_1,V_2]\\
&&F(U_0,U_1,U_2,V_0,V_1,V_2) \mapsto F^{q^2}(U_1,U_2,U_0,V_1,V_2,V_0),
\end{eqnarray*}
and vice versa, where  $F^{q^2}$ denotes the polynomial obtained raising the coefficients of $F$ to the power $q^2$. 

This correspondence yields an effective method to provide a lower bound on the number of $\mathbb{F}_{q^3}$-rational points of $\mathcal{C}_s$. First, we will prove the existence of an absolutely irreducible component of $\mathcal{V}_{s}^{\prime}$ fixed by $\phi$, which  corresponds to an $\mathbb{F}_q$-rational component of $\mathcal{V}_{s}$. This, together with  the celebrated Lang-Weil theorem, yields the desired lower bound.

\begin{theorem}[Lang-Weil Theorem \cite{MR65218}]\label{Th:LangWeil}
Let $\mathcal{V}\subset \mathbb{P}^N(\mathbb{F}_q)$ be an absolutely irreducible variety of dimension $n$ and degree $d$. Then there exists a constant $C$ depending only on $N$, $n$, and $d$ such that 
\begin{equation*}
\left|\#\mathcal{V}(\mathbb{F}_q)-\sum_{i=0}^{n} q^i\right|\leq (d-1)(d-2)q^{n-1/2}+Cq^{n-1}.
\end{equation*}
\end{theorem}

The following result of Cafure and Matera \cite{MR2206396} provides an estimation on the constant $C$ in Theorem \ref{Th:LangWeil}.

\begin{theorem}[\cite{MR2206396}]\label{th:cafmat}
Let $\mathcal{V}\subset \mathbb{A}^N$ be an $\mathbb{F}_q$-irreducible variety of dimension $n$ and degree $d$. For $q>2(n+1)d^2$ holds
$$
|\#\mathcal{V}(\mathbb{F}_q)-q^n|\le (d-1)(d-2)q^{n-1/2}+5d^{\frac{13}{3}}q^{n-1}.
$$
\end{theorem}

In the case $s=1$, the machinery is described in the following proposition.

\begin{proposition}\label{s=1}
Let $s=1$. There are at least $q^3-O(q^{5/2})$  $\mathbb{F}_{q^3}$-rational points on the curve $\mathcal{C}_1$.
\end{proposition}
\begin{proof}
First, we will prove that $\mathcal{V}_1^{\prime}$ contains an  absolutely irreducible component fixed by $\phi(U_0,U_1,U_2,V_0,V_1,V_2)$. 
Consider $\mathcal{V}_1^{\prime\prime}$ defined by 
\begin{equation*}
\begin{cases}
f_1 := {V_0}^2({U_1}^{2}+{U_0})+ {U_0}^2( {V_1}^{2}+ {V_0})=0\\
f_2 := {V_1}^2( {U_2}^{2}+ {U_1})+ {U_1}^2( {V_2}^{2}+ {V_1})=0\\
f_3 := {V_2}^2( {U_0}^{2}+ {U_2})+ {U_2}^2( {V_0}^{2}+ {V_2})=0.
\end{cases}
\end{equation*}
Clearly 
\begin{eqnarray*}
\mathcal{V}_1^{\prime\prime}&=& \mathcal{V}_1^{\prime} \cup (f_2=f_3=U_0+V_0=0) \cup (f_1=f_3=U_1+V_1=0) \cup (f_1=f_2=U_2+V_2=0)\\
&&\cup (f_1=U_1+V_1=U_2+V_2=0) \cup (f_2=U_0+V_0=U_2+V_2=0)\\
&&\cup (f_3=U_0+V_0=U_1+V_1=0)\cup (U_0+V_0=U_1+V_1=U_2+V_2=0).
\end{eqnarray*}
and it is fixed by $\phi(U_0,U_1,U_2,V_0,V_1,V_2)$. 

Let $r_1:= Resultant(f_1,f_2,V_1)$. We have
\begin{eqnarray*}
r_1 &=&U_0^4 U_1^4 V_0 + U_0^4 U_1^4 V_2^4 + U_0^4 U_1^2 V_0^2 + U_0^4 U_2^4 V_0^2 + U_0^3 U_1^4 V_0^2\\  
        &&+U_0^2 U_1^6 V_0^2 + U_0^2 U_1^2 V_0^4 + U_0^2 U_2^4 V_0^4 + U_1^6 V_0^4 + U_1^4 U_2^4 V_0^4.
\end{eqnarray*}
Also, $r_2:= Resultant(r_1,f_3,V_2)$ factorizes as $V_0(V_0+U_0)g(U_0,U_1,U_2,V_0)$, where 
\begin{eqnarray*}
g(U_0,U_1,U_2,V_0) \!\!\!\!&=\!\!\!\!& U_0^{15} U_1^8 V_0 + U_0^{15} U_1^4 V_0^3 + U_0^{15} U_2^8 V_0^3 + U_0^{14} U_1^8 V_0^2 + U_0^{14} U_1^4 V_0^4+U_0^{14} U_2^8 V_0^4\\  
        && + U_0^{13} U_1^4 V_0^5 + U_0^{13} U_2^8 V_0^5 + U_0^{12} U_1^4 V_0^6 + U_0^{12} U_2^8 V_0^6 +    U_0^{11} U_1^{12} V_0^3\\
        &&+ U_0^{10} U_1^{12} V_0^4 + U_0^9 U_1^{12} V_0^5 + U_0^8 U_1^{12} V_0^6 + U_0^7 U_1^8 U_2^8 + U_0^7 U_1^8 U_2^4 V_0\\
        &&+ U_0^7 U_1^6 U_2^8 V_0 + U_0^7 U_1^4 U_2^{12} V_0 + U_0^7 U_1^4 U_2^4 V_0^3 + U_0^7 U_2^{12} V_0^3+ U_0^6 U_1^8 U_2^4 V_0^2\\
        &&+ U_0^6 U_1^6 U_2^8 V_0^2 + U_0^6 U_1^4 U_2^{12} V_0^2 + U_0^6 U_1^4 U_2^4 V_0^4 + U_0^6 U_2^{12} V_0^4 + U_0^5 U_1^{10} U_2^8 V_0\\
        &&+ U_0^5 U_1^4 U_2^4 V_0^5 + U_0^5 U_2^{12} V_0^5 + U_0^4 U_1^{10} U_2^8 V_0^2 + U_0^4 U_1^4 U_2^4 V_0^6 + U_0^4 U_2^{12} V_0^6\\
        &&+ U_0^3 U_1^{12} U_2^4 V_0^3 + U_0^3 U_1^8 U_2^{12} V_0^3 + U_0^2 U_1^{12} U_2^4 V_0^4 + U_0^2 U_1^8 U_2^{12} V_0^4\\ 
        &&+U_0 U_1^{12} U_2^4 V_0^5 + U_0 U_1^8 U_2^{12} V_0^5 + U_1^{12} U_2^4 V_0^6 + U_1^8 U_2^{12} V_0^6
\end{eqnarray*}
Now, $\mathcal{V}_1^{\prime\prime}= \mathcal{W}_1\cup \mathcal{W}_2\cup \mathcal{W}_3$, where 
\begin{eqnarray*}
\mathcal{W}_1&:& \begin{cases}
f_1(U_0,U_1,U_2,V_0,V_1,V_2)=0\\
r_1(U_0,U_1,U_2,V_0,V_2)=0\\
g(U_0,U_1,U_2,V_0)=0, 
\end{cases}\\ 
\mathcal{W}_2&:& \begin{cases}
f_1(U_0,U_1,U_2,V_0,V_1,V_2)=0\\
r_1(U_0,U_1,U_2,V_0,V_2)=0\\
V_0=0,
\end{cases}
\\ 
\mathcal{W}_3&:& \begin{cases}
f_1(U_0,U_1,U_2,V_0,V_1,V_2)=0\\
r_1(U_0,U_1,U_2,V_0,V_2)=0\\
V_0+U_0=0.
\end{cases}
\end{eqnarray*}

By Proposition \ref{W_1}, $\mathcal{W}_1$ is absolutely irreducible. It is readily seen that $\phi$ also fixes $\mathcal{W}_1$, since $\phi(\mathcal{W}_1)$ cannot be contained in $\mathcal{W}_2$ or $\mathcal{W}_3$, by comparing their degrees. Also, $\mathcal{W}_1$ is not contained in $(U_0+V_0=0) \cup (U_1+V_1=0) \cup (U_2+V_2=0)$ and therefore it must be contained in $\mathcal{V}_1^{\prime}$. 

Therefore $\mathcal{W}_1$ corresponds to an $\mathbb{F}_q$-rational absolutely irreducible component of $\mathcal{V}_1$ {via $\phi$}. By Theorem \ref{Th:LangWeil} such a component contains at least $q^3-O(q^{5/2})$ $\mathbb{F}_{q}$-rational points, corresponding to  $q^3-O(q^{5/2})$ $\mathbb{F}_{q^3}$-rational points in $\mathcal{C}_1$.
\end{proof}

The number of $\mathbb{F}_{q^3}$-rational points in $\mathcal{C}_s$ together with an estimate on the maximum number of solutions of $f_{\mu}^{(s)}(x)=0$ will provide the desired result.
\begin{proposition}\label{prop:almost}
Let $f_\mu^{(s)}(x):= x^{2^sq}+\mu x^{2^s}+x$.
Denote by  
$$M_s := \max_{\mu \in \mathbb{F}_{q^3}}\: \ \{\dim_{\F_2}(\ker(f_\mu^{(s)}))\}.$$ Suppose that $\mathcal{C}_s$ contains at least $N_s$ $\mathbb{F}_{q^3}$-rational points. Then, there are at least $1+\frac{N_s}{2^{M_s}-1}$ values of $\mu\in \mathbb{F}_{q^3}$ for which $f_\mu^{(s)}(x)$ is a permutation.
\end{proposition}
\begin{proof}

Let 
$$n_i=\#\{\mu \in \mathbb{F}_{q^3} : \dim_{\F_2}(\ker(f_{\mu}^{(s)}))=i\}.$$

Clearly 
$$\sum_{i=0}^{M_s} n_i=q^3, \quad \sum_{i=1}^{M_s} (2^i-1)n_i=q^3-1, \quad \sum_{i=2}^{M_s} (2^i-1)(2^i-2)n_i\geq N_s.$$
Thus, 
$$n_0=1+\sum_{i=2}^{M_s} (2^i-2)n_i\geq 1+\sum_{i=2}^{M_s} \frac{(2^i-1)(2^i-2)}{2^{M_s}-1}n_i\geq 1+\frac{N_s}{2^{M_s}-1}.$$
The claim follows by observing that, for a fixed $\mu \in \mathbb{F}_{q^3}$,  since $f_\mu^{(s)}(x)$ is a linearized polynomial,  $\dim_{\F_2}(\ker(f_{\mu}^{(s)}))=0$ is equivalent to $f_\mu^{(s)}(x)$ being a permutation.
\end{proof}

We are now in position to prove our main result (for the case $s=1$).
\begin{theorem}\label{Th:s=1}
There are at least 
$1+\frac{q^3-O(q^{5/2})}{7}$  values $\mu\in \mathbb{F}_{q^3}$ for which $f_\mu^{(1)}(x)$ is a permutation.
\end{theorem}
\proof
By Corollary \ref{cor:8}, for $q=2$ and $s=1$, $M_1 := \max_{\mu \in \mathbb{F}_{q^3}}\: \ \{\dim_{\F_2}(\ker(f_\mu^{(1)}))\} \leq 3$. Now Proposition \ref{prop:almost} yields the claim.
\endproof

\begin{corollary}
Let $m\geq 3$. Then, there exists $\mu\in\mathbb{F}_{2^{3m}}$ satisfying $\N_{2^{3m}/2^m}(\mu)\neq 1$ such that $f^{(1)}_\mu(x)$ permutes $\mathbb{F}_{2^{3m}}$. In particular, for any $m\geq 3$ there exists an APN function as in \eqref{eq:APNfun}.
\end{corollary}
\begin{proof}
From Theorem \ref{Th:s=1} and the estimation in Theorem \ref{th:cafmat} applied to the $\mathbb{F}_q$-rational variety  $L(\mathcal{W}_1)$ of degree $1248$ and dimension $3$, we have that for $m\geq 47$ there exists at least one element $\mu\in \F_{2^{3m}}^*$ satisfying $\N_{2^{3m}/2^m}(\mu)\neq 1$ and for which $f^{(1)}_\mu(x)$ permutes $\mathbb{F}_{2^{3m}}$.
For $3\leq m\leq 47$, by a computer check with MAGMA it is possible to obtain an element $\mu$ satisfying these properties. 
\end{proof}

\section*{Acknowledgments}

This research  was supported by the Italian National Group for Algebraic and Geometric Structures and their Applications (GNSAGA - INdAM).
The third and the last authors are supported by the project ``VALERE: VAnviteLli pEr la RicErca" of the University of Campania ``Luigi Vanvitelli''.

\section*{Appendix}

\begin{proposition}\label{W_1}
With the notation as in Proposition \ref{s=1}, $\mathcal{W}_1$ is absolutely irreducible.  
\end{proposition}
\begin{proof}
Let $\mathbb{F}_8^*=\langle \omega \rangle$, where $\omega^3 + \omega + 1=0$. To prove this, it is enough to show that $\mathcal{W}_1\cap (U_0+U_1+U_2=0)\cap (V_0+\omega U_0+\omega^2 U_1=0)$ is an absolutely irreducible curve of the same degree as $\mathcal{W}_1$.  

By direct computations, $\mathcal{W}_1\cap (U_0+U_1+U_2=0)\cap (V_0+\omega U_0+\omega^2 U_1=0)$ reads 
$$\mathcal{W}: \begin{cases}
V_1^2= \frac{U_0^3 +\omega^5 U_0^2 U_1^2 +\omega^5 U_0^2 U_1 +  U_0 U_1^2 + U_1^4}{\omega^3 U_0^2}\\
V_2^4 = \frac{h_1(U_0,U_1)}{\omega^6 U_0^4 U_1^4}\\
h_2(U_0,U_1)=0
\end{cases},$$
where 
\begin{eqnarray*}
h_1(U_0,U_1) \!\!\!\!&=\!\!\!\!& U_0^{10} +\omega^3 U_0^8 U_1^4 +\omega^3 U_0^8 U_1^2 + U_0^6 U_1^2 +\omega^3 U_0^5 U_1^4 +\omega U_0^4 U_1^8 + 
       \omega U_0^4 U_1^6\\
       &&+\omega U_0^4 U_1^5 +  \omega^3 U_0^4 U_1^4 + 
       \omega^3 U_0^3 U_1^6 +\omega U_0^2 U_1^8 + U_0^2 U_1^6 + U_1^{12} + U_1^{10},\\
h_2(U_0,U_1) \!\!\!\!&=\!\!\!\!& 
U_0^{26} +\omega^5 U_0^{25} U_1 +\omega U_0^{24} U_1^2 +\omega U_0^{23} U_1^3 +\omega U_0^{22} U_1^4 + U_0^{22} + \omega^5 U_0^{21} U_1^5\\ &&+\omega^5 U_0^{21} U_1 + U_0^{20} U_1^6 +\omega^6 U_0^{20} U_1^4 +\omega U_0^{20} U_1^2 +\omega^4 U_0^{19} U_1^5+\omega U_0^{19} U_1^3\\
&&+\omega^6 U_0^{18} U_1^6 +\omega U_0^{18} U_1^4 + \omega^5 U_0^{17} U_1^5 +\omega^2 U_0^{16} U_1^6 +\omega^2 U_0^{15} U_1^8+\omega^4 U_0^{15} U_1^7\\
&& + \omega^6 U_0^{14} U_1^{10} +\omega^2 U_0^{14} U_1^8 + U_0^{14} U_1^4 +\omega^4 U_0^{13} U_1^{11}+\omega^5 U_0^{13} U_1^9 +\omega^5 U_0^{13} U_1^5\\
&& +\omega U_0^{12} U_1^{10} +\omega^6 U_0^{12} U_1^8 +\omega U_0^{12} U_1^6+ \omega^4 U_0^{11} U_1^{13} +\omega U_0^{11} U_1^{11}
+\omega^4 U_0^{11} U_1^9\\
&& +\omega U_0^{11} U_1^7 + U_0^{10} U_1^{16}+\omega^6 U_0^{10} U_1^{14} +\omega U_0^{10} U_1^{12} +\omega^6 U_0^{10} U_1^{10} +\omega^3 U_0^{10} U_1^8\\
&&+ \omega^5 U_0^9 U_1^{17} +\omega^5 U_0^9 U_1^{13} +\omega U_0^8 U_1^{18} +\omega^6 U_0^8 U_1^{16} + \omega^2 U_0^8 U_1^{14} +\omega^6 U_0^8 U_1^{12}\\
&&+\omega^3 U_0^8 U_1^{10} +\omega U_0^7 U_1^{19} + \omega^4 U_0^7 U_1^{17} +\omega^2 U_0^7 U_1^{16} +\omega^4 U_0^7 U_1^{15} +\omega^4 U_0^7 U_1^{13}\\
&&+ \omega U_0^7 U_1^{11} +\omega^3 U_0^6 U_1^{20} +\omega^2 U_0^6 U_1^{16} +\omega^6 U_0^6 U_1^{14} + \omega U_0^6 U_1^{12} +\omega^4 U_0^5 U_1^{19}\\
&&+\omega^5 U_0^5 U_1^{17} +\omega^5 U_0^5 U_1^{13} + \omega^3 U_0^4 U_1^{22} +\omega^6 U_0^4 U_1^{20} +\omega U_0^4 U_1^{18} + U_0^4 U_1^{14}+\omega U_0^3 U_1^{23}\\
&& + \omega U_0^3 U_1^{19} +\omega U_0^2 U_1^{24} +\omega U_0^2 U_1^{20} +\omega^5 U_0 U_1^{25} +\omega^5 U_0 U_1^{21} + 
        U_1^{26} + U_1^{22}.
\end{eqnarray*}
MAGMA \cite{MR1484478} shows that $h_2(U_0,U_1)$ is absolutely irreducible and so  is 
$\mathcal{W}_1\cap (U_0+U_1+U_2=0)\cap (V_0+\omega U_0+\omega^2 U_1=0)$ and therefore $\mathcal{W}_1$. {We include below the MAGMA program.}
\end{proof}

{\scriptsize
\begin{verbatim}
s := 1 ;
F<omega>:=GF(8);
K<U_0,U_1,U_2,V_0,V_1,V_2> := PolynomialRing(F,6);
P1 := V_0^(2^s)*(U_1^(2^s)+U_0)+U_0^(2^s)*(V_1^(2^s)+V_0);
P2 := V_1^(2^s)*(U_2^(2^s)+U_1)+U_1^(2^s)*(V_2^(2^s)+V_1);
P3 := V_2^(2^s)*(U_0^(2^s)+U_2)+U_2^(2^s)*(V_0^(2^s)+V_2);

VV := [U_0,U_1,U_1+U_0,omega*U_0+omega^2*U_1,V_1,V_2];

R1 := Resultant(P1,P2,V_1);
R2 := Resultant(R1,P3,V_2);
R2 := Factorization(R2)[3][1];

P1 := Evaluate(P1,VV);
R1 := Evaluate(R1,VV);
R2 := Evaluate(R2,VV);
Factorization(P1);
Factorization(R1);
Factorization(R2);

P<U_0,U_1> := AffineSpace(F,2);

h_2 := U_0^26 + omega^5*U_0^25*U_1 + omega*U_0^24*U_1^2 + omega*U_0^23*U_1^3 + omega*U_0^22*U_1^4 
    + U_0^22 + omega^5*U_0^21*U_1^5 + omega^5*U_0^21*U_1 + U_0^20*U_1^6 + omega^6*U_0^20*U_1^4 +
    omega*U_0^20*U_1^2 + omega^4*U_0^19*U_1^5 + omega*U_0^19*U_1^3 + omega^6*U_0^18*U_1^6 +
    omega*U_0^18*U_1^4 + omega^5*U_0^17*U_1^5 + omega^2*U_0^16*U_1^6 + omega^2*U_0^15*U_1^8 +
    omega^4*U_0^15*U_1^7 + omega^6*U_0^14*U_1^10 + omega^2*U_0^14*U_1^8 + U_0^14*U_1^4 + 
    omega^4*U_0^13*U_1^11 + omega^5*U_0^13*U_1^9 + omega^5*U_0^13*U_1^5 + omega*U_0^12*U_1^10 +
    omega^6*U_0^12*U_1^8 + omega*U_0^12*U_1^6 + omega^4*U_0^11*U_1^13 + omega*U_0^11*U_1^11 +
    omega^4*U_0^11*U_1^9 + omega*U_0^11*U_1^7 + U_0^10*U_1^16 + omega^6*U_0^10*U_1^14 + 
    omega*U_0^10*U_1^12 + omega^6*U_0^10*U_1^10 + omega^3*U_0^10*U_1^8 + omega^5*U_0^9*U_1^17 +
    omega^5*U_0^9*U_1^13 + omega*U_0^8*U_1^18 + omega^6*U_0^8*U_1^16 + omega^2*U_0^8*U_1^14 +
    omega^6*U_0^8*U_1^12 + omega^3*U_0^8*U_1^10 + omega*U_0^7*U_1^19 + omega^4*U_0^7*U_1^17 +
    omega^2*U_0^7*U_1^16 + omega^4*U_0^7*U_1^15 + omega^4*U_0^7*U_1^13 + omega*U_0^7*U_1^11 +
    omega^3*U_0^6*U_1^20 + omega^2*U_0^6*U_1^16 + omega^6*U_0^6*U_1^14 + omega*U_0^6*U_1^12 +
    omega^4*U_0^5*U_1^19 + omega^5*U_0^5*U_1^17 + omega^5*U_0^5*U_1^13 + omega^3*U_0^4*U_1^22 +
    omega^6*U_0^4*U_1^20 + omega*U_0^4*U_1^18 + U_0^4*U_1^14 + omega*U_0^3*U_1^23 + 
    omega*U_0^3*U_1^19+ omega*U_0^2*U_1^24 + omega*U_0^2*U_1^20 + omega^5*U_0*U_1^25 + 
    omega^5*U_0*U_1^21 + U_1^26 + U_1^22;

CC := Curve(P,h_2);
IsAbsolutelyIrreducible(CC);
\end{verbatim}
}


\begin{thebibliography}{10}

\bibitem{diff}
E.~Biham and A.~Shamir.
\newblock Differential cryptanalysis of {DES}-like cryptosystems.
\newblock {\em Journal of CRYPTOLOGY}, 4(1):3--72, 1991.

\bibitem{blu}
A.~W. Bluher.
\newblock On existence of {Budaghyan}--{Carlet} {APN} hexanomials.
\newblock {\em Finite fields and their Applications}, 24:118--123, 2013.

\bibitem{MR1484478}
W.~Bosma, J.~Cannon, and C.~Playoust.
\newblock The {M}agma algebra system. {I}. {T}he user language.
\newblock {\em J. Symbolic Comput.}, 24(3-4):235--265, 1997.
\newblock Computational algebra and number theory (London, 1993).

\bibitem{BBMM}
C.~Bracken, E.~Byrne, N.~Markin, and G.~McGuire.
\newblock A few more quadratic {APN} functions.
\newblock {\em Cryptography and communications}, 3(1):43--53, 2011.

\bibitem{Brack14}
C.~Bracken, C.~H. Tan, and Y.~Tan.
\newblock On a class of quadratic polynomials with no zeros and its application
  to {APN} functions.
\newblock {\em Finite Fields and Their Applications}, 25:26--36, 2014.

\bibitem{dillon}
K.~A. Browning, J.~Dillon, R.~Kibler, and M.~T. McQuistan.
\newblock {APN} polynomials and related codes.
\newblock {\em J. of Combinatorics, Information and System Sciences},
  34(1-4):135--159, 2009.

\bibitem{BCCCV21}
L.~Budaghyan, M.~Calderini, C.~Carlet, R.~Coulter, and I.~Villa.
\newblock Generalized isotopic shift construction for {APN} functions.
\newblock {\em Designs, Codes and Cryptography}, 89(1):19--32, 2021.

\bibitem{BCCCV20}
L.~Budaghyan, M.~Calderini, C.~Carlet, R.~S. Coulter, and I.~Villa.
\newblock Constructing {APN} functions through isotopic shifts.
\newblock {\em IEEE Transactions on Information Theory}, 66(8):5299--5309,
  2020.

\bibitem{BCV20}
L.~Budaghyan, M.~Calderini, and I.~Villa.
\newblock On equivalence between known families of quadratic {APN} functions.
\newblock {\em Finite Fields and Their Applications}, 66:101704, 2020.

\bibitem{BC08}
L.~Budaghyan and C.~Carlet.
\newblock Classes of quadratic {APN} trinomials and hexanomials and related
  structures.
\newblock {\em IEEE Transactions on Information Theory}, 54(5):2354--2357,
  2008.

\bibitem{MR2206396}
A.~Cafure and G.~Matera.
\newblock Improved explicit estimates on the number of solutions of equations
  over a finite field.
\newblock {\em Finite Fields Appl.}, 12(2):155--185, 2006.

\bibitem{Car20}
C.~Carlet.
\newblock {\em Boolean Functions for Cryptography and Coding Theory}.
\newblock Cambridge University Press, 2021.

\bibitem{CCZ}
C.~Carlet, P.~Charpin, and V.~Zinoviev.
\newblock Codes, bent functions and permutations suitable for des-like
  cryptosystems.
\newblock {\em Designs, Codes and Cryptography}, 15(2):125--156, 1998.

\bibitem{CH99}
R.~S. Coulter and M.~Henderson.
\newblock A class of functions and their application in constructing
  semi-biplanes and association schemes.
\newblock {\em Discrete mathematics}, 202(1-3):21--31, 1999.

\bibitem{DO68}
P.~Dembowski and T.~G. Ostrom.
\newblock Planes of order $n$ with collineation groups of order $n^2$.
\newblock {\em Mathematische Zeitschrift}, 103(3):239--258, 1968.

\bibitem{DE14}
U.~Dempwolff and Y.~Edel.
\newblock Dimensional dual hyperovals and {APN} functions with translation
  groups.
\newblock {\em Journal of Algebraic Combinatorics}, 39(2):457--496, 2014.

\bibitem{gologlu}
F.~G{\"o}lo{\u{g}}lu.
\newblock Almost perfect nonlinear trinomials and hexanomials.
\newblock {\em Finite Fields and Their Applications}, 33:258--282, 2015.

\bibitem{MR65218}
S.~Lang and A.~Weil.
\newblock Number of points of varieties in finite fields.
\newblock {\em Amer. J. Math.}, 76:819--827, 1954.

\bibitem{newAPN}
K.~Li, Y.~Zhou, C.~Li, and L.~Qu.
\newblock Two new infinite classes of apn functions.
\newblock {\em arXiv preprint arXiv:2105.08464}, 2021.

\bibitem{Nyb93}
K.~Nyberg.
\newblock Differentially uniform mappings for cryptography.
\newblock In {\em Workshop on the Theory and Application of of Cryptographic
  Techniques}, pages 55--64. Springer, 1993.

\bibitem{polverino2020number}
O.~Polverino and F.~Zullo.
\newblock On the number of roots of some linearized polynomials.
\newblock {\em Linear Algebra and its Applications}, 601:189--218, 2020.

\bibitem{Tan19}
H.~Taniguchi.
\newblock On some quadratic {APN} functions.
\newblock {\em Designs, Codes and Cryptography}, 87(9):1973--1983, 2019.

\end{thebibliography}
\end{document}